\documentclass[11pt]{amsart}
\usepackage{bbm, mathrsfs, enumitem, amssymb, fourier}
\usepackage[foot]{amsaddr}

\def\Fcal{{\mathcal{F}}}

\def\Vcal{{\mathcal{V}}}

\def\ind{\indicator}

\def\E{\mathbb{E}}
\def\P{\mathbb{P}}
\def\R{\mathbb{R}}

\def\X{\mathfrak{X}}

\def\d{\mathrm{d}}

\newcommand{\indicator}[1]{\mathbbm{1}_{\{#1\}}}

\theoremstyle{plain}
\newtheorem{Thm}{Theorem}
\newtheorem{Lem}{Lemma}[section]
\newtheorem{corollary}[Thm]{Corollary}

\theoremstyle{remark}

\title[Tightness for processes with fixed points of discontinuities]{Tightness for processes with fixed points of discontinuities and applications in varying environment}

\author[V.\ Bansaye]{Vincent Bansaye$^1$}
\email{vincent.bansaye@polytechnique.edu}
\address{$^1$CMAP, Ecole Polytechnique, Route de Saclay\\
91128 Palaiseau Cedex, France.}

\author[T.\ Kurtz]{Thomas G.\ Kurtz$^2$}
\email{kurtz@math.wisc.edu}
\address{$^2$Department of Mathematics, University of Wisconsin-Madison, 480 Lincoln Drive, Madison, WI 53706-1388.}

\author[F.\ Simatos]{Florian Simatos$^3$}
\email{florian.simatos@isae.fr}
\address{$^3$ISAE Supaero, D\'epartement DISC, 10 avenue Edouard Belin, BP 54032, 31055 Toulouse Cedex 4, France}

\thanks{This work was funded by project MANEGE 
``Mod\`eles Al\'eatoires en \'Ecologie, G\'en\'etique et \'Evolution'' 
09-BLAN-0215 of ANR (French national research agency), 
Chair Mod\'elisation Math\'ematique et Biodiversit\'e 
VEOLIA-Ecole Polytechnique-MNHN-F.X.  the 
professoral chair Jean Marjoulet, and National Science 
Foundation Grant, DMS 11-06424.  While most of it was 
carried out, F.  Simatos was affiliated with Inria.} 

\date{\today}

\begin{document}

\begin{abstract}
We establish a sufficient condition for the tightness of a 
sequence of stochastic processes.  Our condition makes 
it possible to study processes with accumulations of 
fixed times of discontinuity.  Our motivation comes from 
the study of processes in varying or random 
environment.  We demonstrate the usefulness of our 
condition on two examples:  Galton Watson branching 
processes in varying environment and logistic branching 
processes with catastrophes.  
\end{abstract}

\maketitle

\section{Main result: statement and discussion}

\subsection*{Statement} Let $(\X,d)$ be a separable, 
complete metric space and $D_{\X}$ be the space of c\`adl\`ag 
functions $f:[0,\infty )\to\X$.  The space $D_{\X}$ is endowed with 
the Skorohod $J_1$ topology, and we write $f_n\to f$ for 
convergence in this space and $X_n \Rightarrow X$ for the corresponding weak convergence of stochastic processes. See, for instance, 
Billingsley~\cite{Billingsley99:0} for more details.  For 
$f\in D_{\X}$ and $t\geq 0$ we write $f(t-)=\lim_{s\uparrow t}f(s
)$ (with the 
convention $f(t-)=f(0)$ if $t=0$) and $\Delta f(t)=d(f(t),f(t-))$.  
The above definitions and notation apply to the case 
$\X=\R$ and $d$ is the Euclidean distance, in which case we 
denote by $\Vcal$ the set of c\`adl\`ag functions $f\in D_{\R}$ which 
are non-decreasing.  

For each $n\geq 1$, we consider a c\`adl\`ag process 
$X_n=(X_n(t),t\geq 0)$ adapted to a filtration $\{\Fcal^n_t,t\geq 
0\}$.  

\begin{Thm}\ 
\label{thm} Assume that:  
\begin{enumerate}[label=\textnormal{A\arabic*)}, 
ref=\textnormal{A\arabic*}] 
\item\label{ass:compact-set} For each $T,\varepsilon >0$, there exists 
a compact set $K$ of $\X$ such that 
\begin{equation}\label{eq:compact-set}\liminf_{n\to\infty}\ \P\left
(X_n(t)\in K,\forall t\leq T\right)\geq 1-\varepsilon .\end{equation}
\item\label{ass:oscillations} There exist stochastic 
processes $F_n,F\in\Vcal$ such that $\sigma (F_n)\subset\Fcal^n_0$ and 
$F_n\Rightarrow F$ and
$\beta >0$ such that for every $n\geq 1$ and 
every $0\leq s\leq t$, 

\begin{equation}\label{eq:oscillations}\E\left[1\wedge d\left(X_n
(t),X_n(s)\right)^{\beta}\mid\Fcal^n_s\right]\leq F_n(t)-F_n(s).\end{equation}
\end{enumerate}

Then the sequence $(X_n,n\geq 1)$ is tight in $D_{\X}$.
\end{Thm}

One easily checks, for instance by going back to the 
Arzel\`a--Ascoli characterization of tightness, that in 
presence of the compact containment 
condition~\ref{ass:compact-set} the sequence $(X_n)$ is tight 
if and only if for every compact set $K$, the sequence 
$(X_n)$ stopped upon its first exit of $K$ is tight.  Thus we 
have the following simple extension of the previous 
theorem.  

\begin{corollary} \label{cor:K}
For $K\subset\X$ let $T_n^K=\inf\{t\geq 0:X_n(t)\not \in K\}$.  Assume that the 
compact containment condition~\ref{ass:compact-set} 
holds and that:  
\begin{enumerate}[label=\textnormal{A2')}, 
ref=\textnormal{A2'}] 
\item\label{ass:oscillations-2} For every compact subset 
$K\subset\X$, there exist stochastic processes $F_n,F\in\Vcal$ such 
that $\sigma (F_n)\subset\Fcal_0^n$ and $F_n\Rightarrow F$ and $\beta 
>0$ such that for 
every $n\geq 1$ and every $0\leq s\leq t$, 
\begin{equation}\label{eq:oscillations-2}\E\left[1\wedge d\left(X_
n(t\wedge T^K_n),X_n(s\wedge T^K_n)\right)^{\beta}\mid\Fcal^n_s\right
]\leq F_n(t)-F_n(s).\end{equation}
\end{enumerate}

Then the sequence $(X_n,n\geq 1)$ is tight in $D_{\X}$.
\end{corollary}

We finally mention a second direct extension which is 
useful for the study of Galton--Watson processes in 
varying environments, see below.  

\begin{corollary}\ 
\label{cor:BS} Assume that the compact containment 
condition~\ref{ass:compact-set} holds, and that:  
\begin{enumerate}[label=\textnormal{A2'')}, 
ref=\textnormal{A2''}] 
\item There exist stochastic processes $F_n,F\in\Vcal$ such 
that $\sigma (F_n)\subset\Fcal_0^n$ and $F_n\Rightarrow F$ and $
\beta ,\eta >0$ 
such that for every $n\geq 1$ and every $0\leq s\leq t$ such 
that $F_n(t)-F_n(s)\leq\eta$, 

\begin{equation}\label{eq:oscillations-''}\E\left[1\wedge d\left(
X_n(t),X_n(s)\right)^{\beta}\mid\Fcal^n_s\right]\leq F_n(t)-F_n(s
).\end{equation}
\end{enumerate}

Then the sequence $(X_n,n\geq 1)$ is tight in $D_{\X}$.
\end{corollary}

\begin{proof}
Let $\tilde {F}_n(t)=F_n(t)/1\wedge\eta$: then the inequality

\[\E\left[1\wedge d\left(X_n(t),X_n(s)\right)^{\beta}\mid\Fcal^n_
s\right]\leq\tilde {F}_n(t)-\tilde {F}_n(s)\]
holds for every $0\leq s\leq t$.  Indeed, if $F_n(t)-F_n(s)\leq\eta$ then 
this follows from~\eqref{eq:oscillations-''} by dividing by 
$1\wedge\eta\leq 1$, while if $F_n(t)-F_n(s)\geq\eta$ then $\tilde {
F}_n(t)-\tilde {F}_n(s)\geq 1$ 
and the inequality is trivially satisfied.  Thus we can 
invoke Theorem~\ref{thm} to conclude.  
\end{proof}

\subsection*{Discussion}
If $F$ were continuous, then the 
result would follow immediately from Theorem~$3.8.6$
of~\cite{Ethier86:0} (see also Theorem~$4.20$ 
of~\cite{Kurtz75:0}), but, of course, the point of 
Theorem~$1$ of the paper is that $F$ is not continuous.
Allowing $F$ to be discontinuous is motivated 
by the study of processes in varying environment, 
where, typically, non-critical environments can create 
fixed times of discontinuity which translate to 
discontinuities of $F$.  When there are only finitely many 
fixed times of discontinuity, one can prove tightness on 
time-intervals without fixed times of discontinuity and 
then ``glue'' the pieces together (using for instance 
Lemma~$2.2$ in Whitt~\cite{Whitt80:0}).  However, this 
approach seems more challenging when fixed times of 
discontinuity can accumulate, and even be dense.  The 
interest of Theorem~\ref{thm} is to allow for such 
cases, and we now provide further motivation to study 
this case.

\section{Proof of Theorem~\ref{thm}}

\subsection*{First step}

We start with some preliminary remarks and the 
introduction of some auxiliary functions $\gamma_n$, $Y_n$ and $G_
n$.  
First, note that we can assume without loss of 
generality that $F_n$ and $F$ satisfy the following three 
properties:  
\begin{enumerate}[label={\roman*})] 
\item$F_n(t)-F_n(s),F(t)-F(s)\geq t-s$ for any $0\leq s\leq t$; 
\item$F_n(0)=F(0)=0$; 
\item$F_n$ and $F$ are unbounded.  
\end{enumerate}

Indeed, otherwise we can simply replace $F_n$ and $F$ by 
$\tilde {F}_n(t)=F_n(t)-F_n(0)+t$ and $\tilde {F}(t)=F(t)-F(0)+t$, so that 
$F_n(t)-F_n(s)=\tilde {F}_n(t)-\tilde {F}_n(s)-(t-s)\leq\tilde {F}_
n(t)-\tilde {F}_n(s)$ and 
assumption~\ref{ass:oscillations} still holds with $\tilde {F}_n$ in 
place of $F_n$.  In particular, $F_n$ and $F$ are strictly 
increasing and unbounded.\\

In the sequel, we therefore assume that $F_n$ satisfies 
these three properties.  For $f\in\Vcal$ and unbounded we 
define $f^{-1}\in\Vcal$ the function defined by 
$f^{-1}(t)=\inf\{s\geq 0:f(s)>t\}$.  We will consider in particular 
$\gamma_n=F_n^{-1}$, which satisfies the following properties (see 
Section~$13.6$ in Whitt~\cite{Whitt02:0}):  
\begin{enumerate}[label={\roman*})] 
\item$\gamma_n(0)=0$ and $\gamma_n$ is Lipschitz continuous and 
unbounded; 
\item$\gamma_n^{-1}=F_n$ and $\gamma_n\circ\gamma_n^{-1}=\text{Id}$, with $\text{
Id}$ 
the identity function $\text{Id}(t)=t$;
\item$\gamma_n(t)$ is $\Fcal_0^n$-measurable and hence is a $\{\Fcal_t^n\}$-stopping time.
\end{enumerate}

Note in particular, as a consequence of the Lipschitz 
continuity, that $(\gamma_n)$ is relatively compact.  We further 
define 
\begin{equation}\label{eq:def-Y-G}Y_n(t)=\lim_{s\to t+}X_n\left(\gamma_
n(s)-\right)\ \text{and}\ G_n(t)=\lim_{s\to t+}F_n\left(\gamma_n(
s)-\right).\end{equation}
These strange definitions are used so that we can apply 
Lemma 2.5 of \cite{Kurtz91:1}.  Note that since $F_n$ is 
strictly increasing, it follows that $X_n(t)=Y_n(F_n(t))$.

The compact containment condition for $(Y_n)$ simply 
follows from the identity 
\[\P\left(Y_n(t)\in K,\forall t\leq T\right)=\P\left(X_n(t)\in K,
\forall t\leq\gamma_n(T)\right)\]
together with the facts that the sequence $(\gamma_n(T))$ is 
bounded and that $(X_n)$ satisfies by assumption the 
compact containment condition~\ref{ass:compact-set}.

\subsection*{Second step}

We now prove that the sequence $(Y_n)$ is tight.  Let in 
the sequel $q(x,y)=1\wedge d(x,y)$.  
Note that since $\gamma_n(t)$ is $\Fcal_0^n$-measurable for any $t \geq 0$,~\eqref{eq:oscillations} 
implies that for $0<s<t$ and $0<\delta < \gamma_n(s)$,
\[\E\left[q(X_n(\gamma_n(t)-\delta ),X_n(\gamma_n(s)-\delta ))^{\beta}
|\Fcal_{\gamma_n(s)-\delta}^n\right]\leq F_n(\gamma_n(t)-\delta )
-F_n(\gamma_n(s)-\delta ),\]
and letting $\delta\rightarrow 0$ 
\begin{equation}\E\left[q(X_n(\gamma_n(t)-),X_n(\gamma_n(s)-))^{\beta}
|\Fcal_{\gamma_n(s)-}^n\right]\leq F_n(\gamma_n(t)-)-F_n(\gamma_n
(s)-).\label{predineq}\end{equation}
Again, we are using that for each $t$, $\gamma_n(t)$ is a 
predictable stopping time.  Define $\mathcal{G}_t^n=\cap_{s>t}\Fcal_{
\gamma_n(s)-}^n$. 
Taking decreasing limits in~\eqref{predineq}, we have
\[\E\left[q\left(Y_n(t),Y_n(s)\right)^{\beta}\mid \mathcal{G}^n_s\right
]\leq G_n(t)-G_n(s)\]
which implies in particular that 
$q(Y_n(t),Y_n(s))\leq\indicator{G_n(t)>G_n(s)}$.  Thus for any 
$0\leq v\leq t$ we have 
\[\E\left[q(Y_n(t+u),Y_n(t))^{\beta}\mid\mathcal{G}_t^n\right]q(Y_
n(t),Y_n(t-v))^{\beta}\leq\left(G_n(t+u)-G_n(t)\right)\indicator{
G_n(t)>G_n(t-v)}.\]
Next, Lemma~$2.5$ in Kurtz~\cite{Kurtz91:1} implies that 
$G_n(t)\leq t$ and that if $G_n(t)>G_n(t-v)$, then $G_n(t)>t-v$:  
therefore, 
\[\E\left[q(Y_n(t+u),Y_n(t))^{\beta}\mid\Fcal^n_{\gamma_n(t)}\right
]q(Y_n(t),Y_n(t-v))^{\beta}\leq v+u,\]
where this inequality holds for any $n\geq 1$ and any 
$0\leq v\leq t$ and $u\geq 0$.  These arguments also imply that 
\[\E\left[q(Y_n(\delta ),Y_n(0))^{\beta}\right]\leq G_n(\delta )\leq
\delta ,\]
and these two inequalities imply the desired tightness of 
$(Y_n)$ by Theorem~$3.8.6$ in Ethier and 
Kurtz~\cite{Ethier86:0}, 
since $(Y_n)$ also satisfies the compact containment condition.  

\subsection*{Third step} Let us now conclude the proof 
and show that $(X_n)$ is tight.  Since $\gamma_n\circ\gamma_n^{-1}
=\text{Id}$ 
and $X_n$ is right continuous, $X_n=Y_n\circ F_n$.  Since 
$(Y_n)$ is tight, assume without loss of generality (by 
working along appropriate subsequences and using the
Skorohod representation theorem) that $Y_n\to Y$:  if $Y$ were 
constant (except for maybe one jump) on any interval 
$[u,v]$ on which $F^{-1}$ is constant, then Lemma $2.3$(b) in 
Kurtz~\cite{Kurtz91:1} would imply that $X_n\to Y\circ F$ and 
$(X_n)$ would be tight.  Thus, for each interval $[u,v]$ on which 
$F^{-1}$ is constant, to conclude the proof it is 
enough to show that $Y$ is constant on $[u,v)$.  

Let $\alpha$ denote the constant value taken by $F^{-1}$ on $[u,v
]$, 
and consider  a sequence $(\alpha_n)$ such that $\alpha_n\to\alpha$, 
$F_n(\alpha_n)\to F(\alpha )$ and $F_n(\alpha_n-)\to F(\alpha -)$.  Fix $
u',v'$ with 
$[u',v']\subset (u,v)$.  Since $F$ is constant on $[u,v]$ and takes the 
value $\alpha$, we have $F(\alpha -)\leq u<v\leq F(\alpha )$, and in particular, 
$F_n(\alpha_n-)<u'<v'<F_n(\alpha_n)$ for $n$ large enough.  For these 
$n$, $F_n^{-1}$ is constant on $[u',v']$ and since 
$Y_n(t)=\lim_{s\rightarrow t+}X_n(F^{-1}_n(s)-)$, this 
implies that $Y_n$ for $n$ large enough is constant on $[u',v']$.  
The convergence $Y_n\to Y$ in the Skorohod topology then 
implies that $Y$ is constant on any $[u^{\prime\prime},v^{\prime\prime}
]\subset (u',v')$.  Since 
$u'<v'$ were arbitrary in $[u,v]$, and since $Y$ is c\`adl\`ag, we 
obtain by letting $u^{\prime\prime}\downarrow u$ and $v^{\prime\prime}
\uparrow v$ that $Y$ is constant on 
$[u,v)$ as desired.

\section{Scaling limits of Galton-Watson processes in varying environment}

A Galton Watson branching process (GW process) is an 
integer-valued Markov chain $(Z(k),k\geq 0)$ governed by the 
recursion 
\begin{equation}\label{eq:GW}Z(k+1)=\sum_{i=1}^{Z(k)}\xi_{k,i}\end{equation}
where the $\xi_{k,i}$'s are i.i.d.  random variables having as 
common distribution the so-called \emph{offspring 
distribution}. See, for instance, Athreya and 
Ney~\cite{Athreya04:0} for a  general introduction, and 
the introduction in Bansaye and Simatos~\cite{Bansaye:0} 
for more references pertained to the following 
discussion.  

GW processes in random environments, where the 
sequence of offspring distributions is random, have been 
introduced by Smith and Wilkinson~\cite{Smith69:0} and 
have recently been intensively investigated.  So far, 
they have mostly been studied from the viewpoint of 
their long-time behavior and, as far as we know, their 
scaling limits have only been studied in the finite 
variance case.  This is in sharp contrast with the case 
of constant environment, where scaling limits have been 
exhaustively characterized  by 
Grimvall~\cite{Grimvall74:0}.  This observation was the 
starting point of our investigation in~\cite{Bansaye:0} of 
the scaling limits of GW processes in varying 
environments, where the offspring distribution may 
change from one generation to the next.  This 
corresponds to the quenched approach, where one fixes a 
realization of the sequence of offspring distributions and 
studies the behavior of 
the GW process in this (varying) environment.  \\

In particular, we use Corollary \ref{cor:BS} above in 
order to show in~\cite{Bansaye:0} that the sequence of 
GW processes in a varying environment $(X_n)$ considered is 
tight.  It relies on the domination of a characteristic 
triplet associated to the branching mechanism of $X_n$.  
More precisely, here the process may explode in finite 
time and $[0,\infty ]$ is endowed with the metric 
$d(x,y)=\lvert e^{-x}-e^{-y}\rvert$.  The Assumption A1 is 
automatically satisfied since $[0,\infty ]$ endowed with $d$ is 
compact.  To apply Corollary \ref{cor:BS}, we prove 
in~\cite{Bansaye:0} that  that for each $t\geq 0$, there exists 
$\Delta_t$ such that  for any $s\leq y_0\leq y\leq t$ with $\mu_n
(y_0,y]\leq\Delta_t/2$ 
and $x_0\in [0,\infty ]$, 
\[\E\left[d(x_0,X_n(y))^2\mid X_n(y_0)=x_0\right]\leq 2\Delta_t\mu_
n(y_0,y],\]
where $\mu_n$ is a positive finite measure linked to the 
characteristic triplet of the process $X_n$.  In this case 
Assumption A2'' is satisfied with $\eta =\Delta_t^2$, $F^n=2\Delta_
t\mu_n$ and 
$F=2\Delta_t\mu$.  \\

In this context and in a large population approximation, 
each non-critical offspring distribution (i.e.,  with mean 
not equal to one) induces a deterministic jump in the 
limit:  if $Z(k)$ is large,  then the law of large numbers 
gives, in view of~\eqref{eq:GW}, 
$Z(k+1)-Z(k)\approx\E(\xi_{k,1}-1)Z(k)$.  If the sequence of offspring 
distributions stems from the realization of a sequence of 
i.i.d.  offspring distributions that may be, with positive 
probability, non-critical, then we naturally end up in the 
limit with a time-inhomogeneous Markov process with 
accumulations of fixed times of discontinuity.  Note that 
the possible \emph{accumulations} of these 
discontinuities comes from the fact that, in the usual 
renormalization schemes, time is sped up.  
 
This phenomenon, illustrated on GW processes, is of 
course not unique to this class of processes.  From a 
high-level perspective, it suggests that in a varying 
environment mixing critical and non-critical 
environments, it is natural to expect in the limit 
time-inhomogeneous Markov processes with 
accumulations of fixed times of discontinuity.  For 
instance, the above discussion immediately applies to 
random walks with time-varying step distributions, a 
topic covered by Jacod and Shiryaev~\cite{Jacod03:0}.  It 
is also a very natural framework in population dynamic 
and evolution.  Indeed when considering scaling limits with 
time acceleration in a varying environment, fixed times 
of discontinuity accumulate as soon as instantaneous 
jumps at fixed times are recurrent in the original time 
scale.  In order to illustrate this point, we consider in 
Section~\ref{sec:application} an application of 
Theorem~\ref{thm} to study logistic birth and death 
processes, where the environment provokes catastrophes.

\section{Tightness of logistic branching processes with catastrophes} \label{sec:application}

To further motivate our conditions for tightness, we 
show how to apply the results to the scaling limits of 
logistic branching processes with catastrophes.  

\subsection*{A logistic branching process}  Consider the 
following birth-and-death process:  
\begin{equation}\label{eq:transition-rates}
z\in \{0,1,2,\ldots \}
\longrightarrow\begin{cases}
	z - 1 & \text{ at rate } d z + c z^2,\\
	z + 1 & \text{ at rate } b z,
\end{cases}\end{equation}
for some parameters $b,c,d>0$:  $b$ is the per-individual 
birth rate, $d$ is the per-individual death rate and $c>0$ is 
a logistic term which represents competition between 
individuals.  This process is an example of 
population-dependent branching processes and is also a 
special case of logistic branching processes.  It plays a 
very important role in population dynamics, where it is 
probably the simplest model exhibiting a quasi-stationary 
regime.  Simply put, under a suitable scaling, the 
population size tends to stabilize for a very long time 
around the value $z^{*}=(b-d)/c$ that equalizes the birth 
and death rates.  

Its scaling limits are well-known, namely, if $Z_n$ is the 
above death-and-birth process with parameters $b=\lambda +n\gamma$, 
$d=\mu +n\gamma$ and $c=\kappa /n$, then the renormalized process 
$X_n(t)=Z_n(t)/n$ converges weakly to the logistic Feller 
diffusion, i.e.,  the unique solution to the following 
stochastic differential equation:  
\[\d {X}(t)=\left(\lambda -\mu -\kappa X(t)\right)X(t)\d {t}+\sqrt {
\gamma X(t)}\d {B}(t),\]
with $B$ a standard Brownian motion. See, for 
instance,~\cite{Lambert05:0}.  

\subsection*{A logistic branching process with catastrophes}

There are many different ways to add ``catastrophes'' to 
this logistic branching process.  For example, a common way 
is for the catastrophes to occur
 at the epochs of an independent Poisson 
process, and for each individual to toss a coin and die with a 
certain probability.  However, we adopt a slightly 
different framework, technically more convenient and 
which fulfills our purpose of illustrating the use of 
Theorem~\ref{thm} on a non-trivial example.  Our 
framework comes from the equivalent description of the 
Markov process with transition 
rates~\eqref{eq:transition-rates} via a stochastic 
differential equation,  namely, the unique solution to the 
stochastic differential equation 
\[Z(t)=Z(0)+\int_0^t\int_0^{\infty}\left(\indicator{u\leq bZ(s-)}
-\indicator{bZ(s-)<u\leq (b+d+cZ(s-))Z(s-)}\right)Q(\d {s},\d {u}
),\ t\geq 0,\]
where $Q$ is a Poisson point measure on $[0,\infty )^2$ with 
intensity $\d {s}\times\d {u}$.  A simple generalization 
to this dynamic is given by
\begin{eqnarray}
\label{eq:LBPWC}&&\quad Z(t)=Z(0)+\int_0^t\int_0^{\infty}\left(\indicator{
u\leq bZ(s-)}-\indicator{bZ(s-)<u\leq (b+d+cZ(s-))Z(s-)}\right)Q(
ds,du)\\
&&\qquad\qquad\qquad\qquad -\int_0^t\int_0^1(1-\theta )Z(s-)q(ds,
d\theta )\nonumber\end{eqnarray}
where $q$ is a deterministic point measure on $[0,\infty )\times 
[0,1]$ 
satisfying $q(\{t\}\times [0,1])\in \{0,1\}$ for every $t\geq 0$.  With the 
additional integral term $\int_0^t\int_0^1\left(1-\theta\right)Z(
s-)q(\d {s},\d{\theta })$, if $(t,\theta )$ is 
an atom of $q$, then $Z$ undergoes a catastrophe at time $t$ 
and loses a fraction $\theta\in [0,1]$ of its population.  Note that 
$Z$ given by~\eqref{eq:LBPWC} is no longer integer-valued, 
but this definition will be convenient in order to 
illustrate the use of Theorem~\ref{thm}.  

In the literature, catastrophes are usually added at 
random times, say at the instant of a Poisson process.  
In this case, $q$ would be a Poisson point measure 
independent of $Q$, with intensity $\d {s}\times\P (F\in\d{\theta }
)$ for some 
random variable $F\in [0,1]$:  the above formulation would 
then correspond to the \emph{quenched} approach, 
working conditionally on the random environment.  Let 
us finally mention that this example could be generalized 
in a number of ways, for instance by considering 
positive jumps at fixed times of discontinuity or 
multiple simultaneous births, but here we restrict 
ourselves to the simplest non-trivial example where we 
believe that Theorem~\ref{thm} is useful.  \\

We now consider the same scaling as previously, and we 
write now the birth and death rates for the scaled 
population:  
\begin{equation}\label{eq:scaling}b_n(x)=(\lambda +n\gamma )nx,\quad 
d_n(x)=\left(\mu +n\gamma +\frac {\kappa}nnx\right)nx.
\end{equation}
For each $n\geq 1$, we also consider a measure $q_n$ with 
$q_n(\{t\}\times [0,1])\in \{0,1\}$, 
and we consider $Z_n$ the solution to~\eqref{eq:LBPWC} with these parameters and with initial condition $Z_n(0)=x_0n$ for some $x_0\geq 0$. 
We finally consider the renormalized process
\[ X_n(t) = \frac{Z_n(t)}{n}, \ t \geq 0, \]
which satisfies the following stochastic differential equation:
\begin{eqnarray}
\label{eq:SDE-X}&&\qquad X_n(t)=x_0-\int_0^t\int_0^1\big(1-\theta\big
)X_n(s-)q_n(\d {s},\d{\theta })\\
&&\qquad\quad +\int_0^t\int_0^{\infty}\frac 1n\left(\indicator{u\leq b_
n(X_n(s-))}-\indicator{b_n(X_n(s-))<u\leq d_n(X_n(s-))}\right)Q(\d {
s},\d {u}).\nonumber\end{eqnarray}

Let in the sequel
\[ f_n(t) = \int_0^t \int_0^1 (1-\theta) q_n(\d s, \d \theta) \ \text{ and } \ F_n(t) = t + f_n(t). \]

\begin{Lem} \label{lemma:LBPWC}
For $K\geq 0$, let $T_n^K=\inf\{t\geq 0:X_n(t)\geq K\}$. For any $
T\geq 0$,
\begin{equation}\label{eq:compact-containment-LBPWC}\lim_{K\to\infty}\limsup_{
n\to\infty}\ \P\left(T_n^K\leq T\right)=0,\end{equation}
and for each $K\geq 0$, there exists a constant $C_K$ such that 
the inequality 
\begin{equation}\label{eq:oscillations-LBPWC}\E\left[1\wedge\left
(X_n(t\wedge T_n^K)-X_n(s\wedge T_n^K)\right)^2\mid\Fcal_n(s)\right
]\leq C_K\big(F_n(t)-F_n(s)\big)\end{equation}
holds for all $n\geq 1$ and $0\leq s\leq t$.
\end{Lem}

Assuming that $F_n\to F$ (which holds for instance if $q_n$ 
converges weakly to some measure $q$), this result gives 
the tightness of the sequence $(X_n)$, since the assumptions 
of Corollary~\ref{cor:K} are then satisfied.  Note that $F_n$ 
and its limit $F$ may be discontinuous, and the upper 
bound in~\eqref{eq:oscillations-LBPWC} depends on the 
constant $K$ considered.  Also, it is reasonable when $q_n\to q$ 
to expect any accumulation point to satisfy the following 
stochastic differential equation 
\[\d {X}(t)=\left(\lambda -\mu -cX(t)\right)X(t)\d {t}+\sqrt {\gamma 
X(t)}\d {B}(t)-\int_0^t\int_0^1\big(1-\theta\big)X(s-)q(\d {s},\d{
\theta }).\]

\begin{proof} [Proof of Lemma~\ref{lemma:LBPWC}]
The fact that $(X_n)$ 
satisfies the compact containment 
condition~\eqref{eq:compact-containment-LBPWC} follows 
from a comparison argument:  from~\eqref{eq:scaling} 
and~\eqref{eq:SDE-X} it follows that 
\[X_n(t)\leq X_n(0)+\int_0^t\int_0^{\infty}\frac 1n\left(\indicator{
u\leq nbX_n(s-)}-\indicator{nbX_n(s-)<u\leq n(b+d)X_n(s-)}\right)
Q(\d {s},\d {u})\]
and so classical comparison arguments for stochastic 
differential equations (see for 
instance~\cite[Theorem~V.$43.1$]{Rogers87:0}) imply that 
$X_n(t)\leq\tilde {X}_n(t)$ with $\tilde {X}_n$ given by 
\[\tilde {X}_n(t)=\tilde {X}_n(0)+\int_0^t\int_0^{\infty}\frac 1n\left
(\indicator{u\leq nb\tilde X_n(s-)}-\indicator{nb\tilde X_n(s-)<u
\leq n(b+d)\tilde X_n(s-)}\right)Q(\d {s},\d {u})\]
with $\tilde {X}_n(0)=\lceil nX_n(0)\rceil /n$.  One readily checks that $
\tilde {X}_n$ 
is a linear birth and death process (scaled in time and 
space), whose compact containment condition is easily 
proved (actually, it is well-known that $(\tilde {X}_n)$ converges 
weakly to the Feller diffusion).  We now turn to 
the proof of~\eqref{eq:oscillations-LBPWC}.  The process 
$((X_n(t),t),t\geq 0)$ is Markov with generator 
\begin{eqnarray*}	 
		\Omega_n(f)(x, t) &= & \frac{\partial f}{\partial t}(x, t) + \left( f\left( x + \frac{1}{n}, t \right) - f(x, t) \right) b_n(x) 
		+ \left( f\left( x - \frac{1}{n}, t \right) - f(x, t) \right) d_n(x)\\
		&& \hspace{10mm} + \int q_n(\{t\} \times \d \theta) \left( f(\theta x, t) - f(x, t) \right)
\end{eqnarray*}
and so the stopped process $((X_n(t\wedge T^K_n),t\wedge T^K_n),t\geq 0)$ 
is Markov and its generator is given by $\Omega_n(f)(x,t)\indicator{x\leq K}$. 
In particular, for a function $f$ that only depends on $x$,
 defining $X^K_n(t)=X_n(t\wedge T^K_n)$,
\begin{eqnarray*}
\E\left[f(X^K_n(t))\right]&&=f(X^K_n(0))\\
&&\quad +\int_0^t\E\left[\left(f\left(X^K_n(s)+\frac 1n\right)-f(
X^K_n(s))\right)b_n(X^K_n(s))\ind{T^K_n>s}\right]\d {s}\\
&&\quad +\int_0^t\E\left[\left(f\left(X^K_n(s)-\frac 1n\right)-f(
X^K_n(s))\right)d_n(X^K_n(s))\ind{T^K_n>s}\right]\d {s}\\
&&\quad +\int_0^t\int q_n(\d {s}\times\d{\theta })\E\left[\left(f
(\theta X^K_n(s))-f(X^K_n(s))\right)\ind{T^K_n>s}\right].\end{eqnarray*}
For $f(x)=(x-X^K_n(0))^2$, we find after some computation
\begin{eqnarray*}
&&\E\left[\left(X^K_n(t)-X^K_n(0)\right)^2\mid X^K_n(0)\right]\\
&&\qquad =2(\lambda -\mu )\int_0^t\E\left[X^K_n(s)(X^K_n(s)-X^K_n
(0))\ind{T^K_n>s}\right]\d {s}\\
&&\qquad\qquad +\frac {\lambda +\mu +2\gamma n}n\int_0^t\E\left(X^
K_n(s)\ind{T^K_n>s}\right)\d {s}\\
&&\qquad\qquad +\frac cn\int_0^t\E\left(X^K_n(s)^2\ind{T^K_n>s}\right
)\d {s}\\
&&\qquad\qquad -2c\int_0^t\E\left[X^K_n(s)^2(X^K_n(s)-X^K_n(0))\ind{
T^K_n>s}\right]\d {s}\\
&&\qquad\qquad -\int_0^t\int q_n(\d {s}\times\d{\theta })(1-\theta^
2)\E\left[X^K_n(s)^2\ind{T^K_n>s}\right]\\
&&\qquad\qquad +\int_0^t\int q_n(\d {s}\times\d{\theta })2(1-\theta 
)\E\left[X^K_n(s)X^K_n(0)\ind{T^K_n>s}\right].\end{eqnarray*}
Since
$f(X_n(t))=0$ for $X_n(0)>K$, we can assume that $X_n(0)\leq K$ 
and we get 
\begin{eqnarray*}
\E\left[\left(X^K_n(t)-X^K_n(0)\right)^2\mid X^K_n(0)\right]\leq 
4\lvert\lambda -\mu\rvert K^2t+\frac {(\lambda +\mu +2\gamma n)Kt}
n+\frac cnK^2t+2cK^3t\\
+2K^2\int_0^t\int (1-\theta )q_n(\d {s}\times\d{\theta }).\end{eqnarray*}

Since all sequences involved  are bounded, the result follows.
\end{proof}

\end{document}